\documentclass[11pt,leqno]{amsart}
\usepackage{amsmath,amsfonts,amssymb,amsthm}
\usepackage{graphicx}

\theoremstyle{plain}
\newtheorem{theorem}{Theorem}
\newtheorem{corollary}[theorem]{Corollary}

\newtheorem{proposition}[theorem]{Proposition}

\theoremstyle{definition}

\newtheorem{remark}[theorem]{Remark}

\newcommand{\R}{\mathbb R}

\newcommand{\cL}{\mathcal L}

\newcommand{\loc}{{\scriptstyle{loc}}}

\DeclareMathOperator{\diver}{div}

\numberwithin{theorem}{section} \numberwithin{equation}{section}

\title{Superposition of fundamental solutions of second order quasilinear equations}
\author{Jeremy T. Tyson}
\address{Department of Mathematics \\ University of Illinois at Urbana-Champaign \\ 1409 West Green St. \\ Urbana, IL 61801}
\date{\today}
\thanks{The author was supported by NSF Grant DMS-1600650 `Mappings and measures in sub-Riemannian and metric spaces'.}

\begin{document}

\maketitle

\begin{abstract}
We prove a superposition principle in the spirit of Crandall--Zhang and Lindqvist--Manfredi for a class of second order quasilinear equations. For suitable $\alpha$, the $\alpha$-Riesz potentials of nonnegative and compactly supported continuous functions are either subsolutions or supersolutions for the operator $\cL_{p,q} u = \diver(u^{q-1}|\nabla u|^{p-2}\nabla u)$. This class of operators includes both the $p$-Laplacian as well as the stationary porous medium equation.
\end{abstract}

\section{Introduction}

In this paper we establish a superposition principle for supersolutions of a class of second order quasilinear equations. The equation which we consider includes both the $p$-Laplace equation and the stationary porous medium equation as special cases. The aforementioned superposition principle was first established for the $p$-Laplace equation by Crandall and Zhang \cite{cz:another}; see also Lindqvist and Manfredi \cite{lm:remarkable} for an alternate approach and more general conclusions. The observation that such a superposition principle applies also to stationary solutions of the porous medium equation is new and was the initial impetus for this paper.

As a point of departure let us recall the role of the Riesz potential
\begin{equation}\label{eq:Riesz}
I_\alpha(\rho)(x) = c_{n,\alpha} \int_{\R^n} \frac{\rho(y)}{|x-y|^{n-\alpha}} \, dy, \qquad c_{n,\alpha} = \frac{\Gamma\left(\tfrac{n-\alpha}2\right)}{\pi^{n/2}2^\alpha\Gamma(\tfrac\alpha2)},
\end{equation}
in classical potential theory and harmonic analysis. When $n\ge 3$ the operator $I_2$ acts as inverse of the Laplacian: if $\rho$ is a nonnegative and compactly supported continuous function then $(-\triangle)(I_2(\rho)) = \rho$. In particular, $I_2(\rho)$ is a nonnegative superharmonic function. A discrete analog of the preceding observation is the fact that nonnegative superpositions of fundamental solutions,
$$
u(x) = \sum_{j=1}^N A_j |x-a_j|^{2-n}, \qquad a_j \in \R^n, \, A_j \ge 0,
$$
are superharmonic.

In \cite{cz:another}, Crandall and Zhang made a surprising discovery: the preceding superposition principle holds also for the nonlinear $p$-Laplace equation
$$
\triangle_p u = \diver(|\nabla u|^{p-2}\nabla u), \qquad 1<p<\infty.
$$
Specifically, they showed that linear combinations of the form
$$
u(x) = \sum_{j=1}^N A_j |x-a_j|^{(p-n)/(p-1)}
$$
or
$$
u(x) = \sum_{j=1}^N A_j \log(1/|x-a_j|)
$$
with $a_j \in \R^n$ and $A_j \ge 0$, are $p$-sub- or $p$-superharmonic depending on the choice of $p>2$ and $n$. Recall that the fundamental solution of the $p$-Laplace operator $-\triangle_p$ is given by a multiple of the function $x \mapsto |x|^{(p-n)/(p-1)}$ or $x \mapsto \log(1/|x|)$ when $1<p<\infty$.

Lindqvist and Manfredi \cite{lm:remarkable} observed that the discrete result extends also to the continuous (Riesz potential) setting, and generalized the results of Crandall and Zhang to allow for other exponents besides $(p-n)/(p-1)$. The main result of \cite{lm:remarkable} reads as follows.

\begin{theorem}[Lindqvist--Manfredi]\label{th:lm}
Let $\rho$ be a nonnegative and compactly supported continuous function on $\R^n$, $n \ge 3$.
\begin{enumerate}
\item If $2\le p < n$, then $I_{n-\alpha}(\rho)$ is $p$-superharmonic if $0<\alpha\le\tfrac{n-p}{p-1}$.
\item If $p>n$, then $I_{n-\alpha}(\rho)$ is $p$-subharmonic if $-\alpha \ge \tfrac{p-n}{p-1}$. Moreover, $I_{n-\alpha}(\rho)$ is $\infty$-subharmonic if $-\alpha \ge 1$.
\item If $p=n$ then then function $I_n(\rho)(x) = \int_{\R^n} \rho(y) \, \log|x-y| \, dy$ is $n$-subharmonic.
\end{enumerate}
\end{theorem}
Note the restriction to $p \ge 2$. We refer the reader to \cite[p.\ 134]{lm:remarkable} for a detailed discussion of the role of this assumption.

A related superposition principle for elliptic second order partial differential operators was considered by Laugesen and Watson in \cite{lw:green}. In \cite{gt:h-type} Theorem \ref{th:lm} was generalized to the setting of step two Carnot groups of Heisenberg type, for the subelliptic $p$-Laplace equation. In that setting the relevant Riesz potentials are defined by convolving with radial kernels defined using Kaplan's anisotropic homogeneous norm.

The starting point for this paper was the observation that a similar superposition principle holds for stationary solutions of the porous medium equation (PME). That is, if we set
$$
\cL_q u = \diver(\nabla(u^q)), \qquad q>0,
$$
then Riesz potentials $I_{n-\alpha}(\rho)$ of nonnegative and compactly supported continuous functions $\rho$ are either subsolutions or supersolutions for the operator $-\cL_q$ depending on the choice of $q$, $n$ and $\alpha$. This assertion is a special case of our main theorem (Theorem \ref{th:main}) stated below. The theory of the porous medium equation
$$
u_t = \diver(\nabla(u^q))
$$
is comprehensively developed in the books by V\'azquez \cite{v:book} and DiBenedetto \cite{db:book}. The PME is a standard model for a variety of physically relevant systems, including flow of gas through a porous medium and heat radiation in plasmas.

We will establish such a superposition principle for a class of second order quasilinear equations which includes both the $p$-Laplace equation and the stationary porous medium equation. Throughout this paper we fix
\begin{equation}\label{standing-assumptions}
n \ge 3, \quad q>0, \quad \mbox{and} \quad p \ge 2
\end{equation}
and consider the operator
\begin{equation}\label{eq:L}
\cL_{p,q} u := \diver(u^{q-1}|\nabla u|^{p-2} \nabla u)
\end{equation}
acting in $\R^n$. This is the stationary form of the equation which is called the {\it doubly nonlinear diffusion equation} in \cite{v:book}. Manfredi and Vespri \cite{mv:doubly}, in their study of the asymptotics of solutions of the doubly nonlinear diffusion equation, write `it seems interesting to see if and how many of the properties of the solutions of the porous media and $p$-Laplacian equations are preserved in this more general case'. This paper contributes to such a program of research.

The operator $\cL_{p,q}$ is closely related to the $p$-Laplace operator $\triangle_p$. For $p$ and $q$ as above we set
\begin{equation}\label{m}
m = \frac{p-2+q}{p-1}
\end{equation}
and note that $m>0$. A straightforward computation, valid for nonnegative $C^2$ functions $u$, gives
\begin{equation}\label{p-Laplace-and-L-p-q}
\triangle_p(u^m) = m^{p-1} \, \cL_{p,q}(u).
\end{equation}
By applying \eqref{p-Laplace-and-L-p-q}, or by a direct computation, one observes that the fundamental solution for $-\cL$ is a multiple of
\begin{equation}\label{eq:L-fs}
x \mapsto |x|^{(p-n)/(p-2+q)}
\end{equation}
or $x \mapsto \log(1/|x|)$ if $p=n$. See Proposition \ref{prop:fs} for a precise statement.

A function $u$ is said to be a {\it (weak) supersolution} for the operator $-\cL_{p,q}$ if
\begin{equation}\label{integrability-assumptions}
|u| \in L^{mp}_\loc, \quad |u|^{m-1}\,|\nabla u| \in L^p_\loc,
\end{equation}
and for all nonnegative functions $\varphi \in C^\infty_0(\Omega)$,
$$
\int_\Omega |u|^{q-1} |\nabla u|^{p-2} \langle \nabla u,\nabla \varphi \rangle \ge 0.
$$
Compare, e.g., \cite{mv:doubly}. Subsolutions are defined similarly by reversing the inequalities. With the above definition in place it is easy to confirm that $u$ is a nonnegative supersolution for the operator $-\cL_{p,q}$ if and only if $u^m$ is a nonnegative supersolution for the $p$-Laplace operator $-\triangle_p$.

The main result of this paper is the following.

\begin{theorem}\label{th:main}
Let $\rho$ be a nonnegative and compactly supported continuous function in $\R^n$.
\begin{enumerate}
\item If $2\le p<n$ and $q\ge 1$ then
\begin{equation}\label{eq:main-u}
u:=I_{n-\alpha}(\rho)
\end{equation}
is a supersolution for the operator $-\cL_{p,q}$ if
$$
0<\alpha \le \frac{n-p}{p-2+q}.
$$
\item If $p>n$ and $0<q\le 1$ then $u:=I_{n-\alpha}(\rho)$ is a subsolution for $-\cL_{p,q}$ if
$$
-\alpha \ge \frac{p-n}{p-2+q}.
$$
If $p=\infty$ we may choose $-\alpha \ge 1$.
\item If $p=n$ and $q>0$ then $u(x):=\int \log|x-y| \, \rho(y) \, dy$ is a supersolution for $-\cL_{n,q}$.
\end{enumerate}
\end{theorem}

In view of the regularity properties of the function $u$ in the theorem, the conclusion can also be stated by saying that $u^m$ (with $m$ as in \eqref{m}) is either $p$-superharmonic or $p$-subharmonic. By an approximation argument, the latter conclusion continues to hold also for Riesz potentials of general Radon measures satisfying a suitable growth condition. See Theorem \ref{th:radon}. Applying the latter conclusion to finite sums of Dirac masses, or by repeating the proof of Theorem \ref{th:main} in the discrete setting, we obtain the following corollary.

\begin{corollary}\label{cor:main}
Let
$$
u(x) = \sum_{j=1}^N A_j |x-a_j|^{-\alpha}
$$
for points $a_1,\ldots,a_N \in \R^n$ and nonnegative coefficients $A_1,\ldots,A_N$. Then the following conclusions hold.
\begin{enumerate}
\item If $2\le p<n$, $q\ge 1$ and $\tfrac{p-n}{p-2+q} \le \alpha < 0$ then $u^m$ is $p$-superharmonic.
\item If $p>n$, $0<q\le 1$ and $-\alpha \ge \tfrac{p-n}{p-2+q}$ then $u^m$ is $p$-subharmonic. If $p=\infty$ and $0<q\le 1$ then we may choose any $\alpha$ satisfying $-\alpha \ge 1$.
\item If $p=n$ and $q>0$ then $u(x) = \sum_j A_j \log(1/|x-a_j|)$ is $n$-superharmonic.
\end{enumerate}
\end{corollary}

Note that linear combinations of fundamental solutions as in Corollary \ref{cor:main} need not be supersolutions for the operator $-\cL_{p,q}$. Indeed, even the fundamental solution itself does not satisfy the integrability conditions \ref{integrability-assumptions}.

Let us briefly comment on the proof of Theorem \ref{th:main}. As previously mentioned, we first established this result for the stationary porous medium equation. The argument in that case is similar to, but somewhat easier than, the argument in the $p$-Laplace case. In keeping with the overall theme of the paper, our proof of Theorem \ref{th:main} can loosely be described as `a superposition of the $p$-Laplace and stationary PME proofs'. We explain in detail the meaning of this assertion in section \ref{sec:proof}. We expand $\cL_{p,q} u$ for the function $u$ in \eqref{eq:main-u} and write the result as a sum of various expressions of fixed sign, each depending on the values of $n$, $p$ and $q$. This yields the desired conclusion. In order to identify the desired expressions we carefully decompose the expanded form of $\cL_{p,q} u$ into two parts, which we treat independently by the methods used in the $p$-Laplace and stationary PME cases. The main subtlety in the proof lies in showing that this scheme is effective, specifically, that the two terms which we obtain have the same sign. This constraint leads to the restrictions on $p$ and $q$ indicated in the statement of the theorem.

This paper is organized as follows. In section \ref{sec:fs} we identify the fundamental solution for the operator $\cL_{p,q}$. In section \ref{sec:proof} we give the proof of Theorem \ref{th:main} and its counterpart (Theorem \ref{th:radon}) for general Radon measures.

\section{A fundamental solution for the operator $\cL_{p,q}$}\label{sec:fs}

In this section we compute the fundamental solution of the operator $\cL_{p,q}$ defined in \eqref{eq:L}. We first record the following expansion, which is valid for $C^2$ functions $u$:
$$
\cL_{p,q} u = u^{q-2} |\nabla u|^{p-4} \left( u|\nabla u|^2\triangle u + (p-2) u \triangle_\infty u + (q-1) |\nabla u|^4 \right).
$$
Here $\triangle_\infty u$ denotes the $\infty$-Laplacian of $u$, i.e.,
$$
\triangle_\infty u = \tfrac12 \langle \nabla |\nabla u|^2 , \nabla u \rangle = \sum_{j,k} u_j u_k  u_{j,k},
$$
where we have used subscripts to denote partial differentiation.

\begin{proposition}\label{prop:fs}
For $n \ge 3$, $p \ge 2$ and $q>0$, set
\begin{equation}\label{gamma}
\gamma = \frac{p-n}{p-2+q}.
\end{equation}
The fundamental solution of the operator $-\cL_{p,q}$ is given by
$$
u_0(x) = c(n,p,q) |x|^\gamma
$$
if $p \ne n$, for a suitable choice of the constant $c(n,p,q)$. If $p=n$ the fundamental solution is a multiple of $\log(1/|x|)$.
\end{proposition}

\begin{proof}
We first show that $\cL_{p,q} u_0 = 0$ in the complement of the origin. We compute
$$
\nabla u = -\gamma |x|^{-\gamma-2} x,
$$
$$
|\nabla u|^2 = \gamma^2 |x|^{-2\gamma-2},
$$
$$
\triangle u = \gamma(\gamma+2-n)|x|^{-\gamma-2},
$$
and
$$
\triangle_\infty u = \gamma^3(\gamma+1)|x|^{-3\gamma-4}.
$$
Hence
\begin{equation*}\begin{split}
&u|\nabla u|^2\triangle u + (p-2) |\nabla u|^{p-4} \triangle_\infty u + (q-1) |\nabla u|^4 \\
& \qquad = \left\lbrace \gamma^3(\gamma+2-n)+\gamma^3(\gamma+1)(p-2)+\gamma^4(q-1) \right\rbrace |x|^{-4\gamma-4}
\end{split}\end{equation*}
and we note that the expression inside the braces vanishes precisely when $\gamma$ is as in \eqref{gamma}.

To complete the proof we show that an appropriate choice of the constant $c(n,p,q)$ ensures that
$$
-\cL_{p,q} u_0 = \delta_0
$$
in the sense of distributions, where $\delta_0$ denotes the Dirac distribution with pole at the origin. Integrating $-\cL_{p,q} u_0$ by parts against a test function $\varphi \in C^\infty_0(\R^n)$ yields
$$
c(n,p,q)^{p+q-2} \gamma^{p-1} \int_{\R^n} |x|^{\gamma(p+q-2)-p+1} \langle x,\nabla \varphi(x)\rangle \, dx.
$$
Substituting the value of $\gamma$ indicated in \eqref{gamma} gives
\begin{equation}\label{eq:weak-L}
c(n,p,q)^{p+q-2} \gamma^{p-1} \int_{\R^n} |x|^{1-n} \langle x,\nabla \varphi(x)\rangle \, dx
\end{equation}
and we recognize the integral in \eqref{eq:weak-L} as coinciding with the action of the distribution $-\triangle((2-n)^{-1}|x|^{2-n})$ on $\varphi$. An explicit value for the constant $c(n,p,q)$ can be computed from \eqref{eq:weak-L} and the value of the constant $c_{n,2}$ in \eqref{eq:Riesz}.
\end{proof}

\begin{remark}
Proposition \ref{prop:fs} can also be derived by using the relationship between solutions for $\cL_{p,q}$ and powers of solutions for $\triangle_p$, and the known fundamental solution for the $p$-Laplacian.
\end{remark}

\section{Proof of Theorem \ref{th:main}}\label{sec:proof}

Since $p \ge 2$, supersolutions $u$ for the operator $-\cL$ are characterized by the pointwise inequality $-\cL u \ge 0$. We will compute the expression
\begin{equation}\label{term-in-L}
I := u|\nabla u|^2\triangle u + (p-2) u \triangle_\infty u + (q-1) |\nabla u|^4
\end{equation}
and will show that it has the correct sign depending on the values of $p$ and $q$ as indicated. Note that $\gamma>0$ in the first case, while $\gamma<0$ in the second case.

We drop the immaterial constant $c_{n,\alpha}$ and consider
$$
u(x) = \int \frac{\rho(y)\,dy}{|x-y|^\alpha}
$$
in cases (1) and (2). Case (3) (the logarithmic case) is similar and will be left to the reader.

We will compute formally by differentiating under the integral sign. The ensuing computations are justified by the following remarks. First, since $\rho$ is compactly supported there are no convergence problems at infinity. Moreover, the singularity at $y=x$ is sufficiently mild so that the interchange of limits is justified. In fact, we always have $\alpha < n-2$. Note that $\alpha<0$ in case (2) so it suffices to consider case (1) ($2\le p<n$ and $q\ge 1$). But in that case we have
\begin{equation*}\begin{split}
n-2-\alpha &\ge (n-2) - \frac{n-p}{p-2+q} \\ &= \frac{n(p+q-3)-(p+2q-4)}{p-2+q} \\ &> \frac{2(p+q-3)-(p+2q-4)}{p-2+q} = \frac{p-2}{p-2+q} \ge 0.
\end{split}\end{equation*}
The preceding remarks understood we proceed to compute various derivatives of $u$. First,
$$
u_j(x) = -\alpha \int \frac{(x_j-y_j)\rho(y)\,dy}{|x-y|^{\alpha+2}}
$$
and
$$
\nabla u(x) = -\alpha \int \frac{(x-y) \rho(y)\,dy}{|x-y|^{\alpha+2}}.
$$
Thus
$$
|\nabla u(x)|^2 = \alpha^2 \iint \frac{\langle x-y,x-z\rangle \rho(y)\rho(z)\,dy\,dz}{|x-y|^{\alpha+2}|x-z|^{\alpha+2}}.
$$
Next we compute second derivatives
$$
u_{j,k}(x) = -\alpha \int \frac{\delta_{jk} \rho(y)\,dy}{|x-y|^{\alpha+2}} + \alpha(\alpha+2) \iint \frac{(x_j-y_j)(x_k-y_k)\rho(y)\,dy}{|x-y|^{\alpha+4}}
$$
where $\delta_{jk}$ denotes the Kronecker delta. The Laplacian of $u$ is then given by
$$
\triangle u = \alpha(\alpha+2-n) \int \frac{\rho(y)\,dy}{|x-y|^{\alpha+2}}
$$
and the $\infty$-Laplacian of $u$ is given by
\begin{equation*}\begin{split}
\triangle_\infty u
&= \sum_{j,k} \left( -\alpha^3 \iiint \frac{\delta_{jk}(x_j-y_j)(x_k-z_k)\rho(y)\rho(z)\rho(v)\,dy\,dz\,dv}{|x-y|^{\alpha+2}|x-z|^{\alpha+2}|x-v|^{\alpha+2}} \right. \\
& \quad \left. + \alpha^3(\alpha+2) \iiint \frac{(x_j-y_j)(x_k-z_k)(x_j-v_j)(x_k-v_k) \rho(y)\rho(z)\rho(v)\,dy\,dz\,dv}{|x-y|^{\alpha+2}|x-z|^{\alpha+2}|x-v|^{\alpha+4}} \right) \\
&= -\alpha^3 \iiint \frac{\langle x-y,x-z\rangle \rho(y)\rho(z)\rho(v)\,dy\,dz\,dv}{|x-y|^{\alpha+2}|x-z|^{\alpha+2}|x-v|^{\alpha+2}} \\
& \quad + \alpha^3(\alpha+2) \iiint \frac{\langle x-y,x-v\rangle \, \langle x-z,x-v\rangle \rho(y)\rho(z)\rho(v)\,dy\,dz\,dv}{|x-y|^{\alpha+2}|x-z|^{\alpha+2}|x-v|^{\alpha+4}}
\end{split}\end{equation*}
The expression $I$ in \eqref{term-in-L} can now be written as a sum of quadruple integrals:
\begin{equation*}\begin{split}
I &= \alpha^3(\alpha+4-n-p) \iiiint \frac{\langle x-z,x-v\rangle dV}{|x-y|^\alpha |x-z|^{\alpha+2} |x-v|^{\alpha+2} |x-w|^{\alpha+2}} \\
& \quad + (p-2)\alpha^3(\alpha+2) \iiiint \frac{\langle x-z,x-w\rangle \, \langle x-v,x-w\rangle
dV}{|x-y|^\alpha |x-z|^{\alpha+2} |x-v|^{\alpha+2} |x-w|^{\alpha+4}} \\
& \quad + (q-1) \alpha^4 \iiiint \frac{\langle x-y,x-w\rangle \, \langle x-z,x-v \rangle
dV}{|x-y|^{\alpha+2} |x-z|^{\alpha+2} |x-v|^{\alpha+2} |x-w|^{\alpha+2}} \\
& = I_1 + I_2 + I_3
\end{split}\end{equation*}
where we abbreviated $dV = \rho(y)\rho(z)\rho(v)\rho(w) \, dy \, dz \, dv \, dw$.

Observe that each of the quadruple integrals in the preceding equation is nonnegative. Indeed, the first quadruple integral is equal to
$$
\int \frac{\rho(y)\,dy}{|x-y|^\alpha} \int \frac{\rho(w)\,dw}{|x-w|^{\alpha+2}} \left| \int \frac{(x-z)\rho(z)\,dz}{|x-z|^{\alpha+2}} \right|^2
$$
and the other two expressions can be written in similar fashion. If $\alpha \le -2$ then all three coefficients in front of these integrals are nonnegative and hence $I \ge 0$. This means that in case (2) of the main theorem we may assume without loss of generality that
\begin{equation}\label{eq:alpha-restriction}
\alpha > -2.
\end{equation}
We now fix
\begin{equation}\label{eq:lambda}
\lambda = -(p-2)(\alpha+2),
\end{equation}
noting by \eqref{standing-assumptions} and \eqref{eq:alpha-restriction} that $\lambda\le 0$,
and we write $I_1$ as the sum of two terms $I_{1,1}$ and $I_{1,2}$, where $I_{1,1}$ has the coefficient $\alpha+4-n-p$ replaced by $\alpha+4-n-p-\lambda$ and $I_{1,2}$ has that same coefficient replaced by $\lambda$. In $I_{1,1}$ we interchange the roles of the variables $y$ and $w$ and average the result with $I_{1,1}$ itself. This does not affect the value of $I_{1,1}$. We then write
$$
I = (I_{1,1} + I_3) + (I_{1,2} + I_2)
$$
where
\begin{equation}\begin{split}\label{I11I3}
&I_{1,1} + I_3 \\
&\quad = \tfrac12\alpha^3(\alpha+4-n-p-\lambda) \iiiint \frac{\langle x-z,x-v\rangle dV}{|x-y|^\alpha |x-z|^{\alpha+2} |x-v|^{\alpha+2} |x-w|^{\alpha+2}} \\
&\qquad + \tfrac12\alpha^3(\alpha+4-n-p-\lambda) \iiiint \frac{\langle x-z,x-v\rangle dV}{|x-y|^{\alpha+2} |x-z|^{\alpha+2} |x-v|^{\alpha+2} |x-w|^{\alpha}} \\
&\quad\qquad + (q-1) \alpha^4 \iiiint \frac{\langle x-y,x-w\rangle \, \langle x-z,x-v \rangle
dV}{|x-y|^{\alpha+2} |x-z|^{\alpha+2} |x-v|^{\alpha+2} |x-w|^{\alpha+2}} \\
\end{split}\end{equation}
and
\begin{equation}\begin{split}\label{I12I2}
&I_{1,2} + I_2 \\
& \quad = \alpha^3 \lambda \iiiint \frac{\langle x-z,x-v\rangle dV}{|x-y|^\alpha |x-z|^{\alpha+2} |x-v|^{\alpha+2} |x-w|^{\alpha+2}} \\
& \qquad + (p-2)\alpha^3(\alpha+2) \iiiint \frac{\langle x-z,x-w\rangle \, \langle x-v,x-w\rangle dV}{|x-y|^\alpha |x-z|^{\alpha+2} |x-v|^{\alpha+2} |x-w|^{\alpha+4}} \\
& \quad = \alpha^3 \lambda \iiiint \frac{(|x-w|^2 \, \langle x-z,x-v\rangle - \langle x-z,x-w\rangle \, \langle x-v,x-w\rangle) \, dV}{|x-y|^\alpha |x-z|^{\alpha+2} |x-v|^{\alpha+2} |x-w|^{\alpha+2}}.
\end{split}\end{equation}
Let us first analyze the expression in \eqref{I12I2}. We recall the identity $|U|^2\langle V,W\rangle - \langle U,V\rangle \, \langle U,W\rangle = \langle U\wedge V,U\wedge W \rangle$, where we have introduced the induced inner product on the space $\Lambda^2\R^n$ of alternating two-vectors. We obtain
$$
I_{1,2}+I_2 = \alpha^3 \lambda \iiiint \frac{\langle (x-z)\wedge(x-w),(x-v)\wedge(x-w) \rangle \, dV}{|x-y|^\alpha |x-z|^{\alpha+2} |x-v|^{\alpha+2} |x-w|^{\alpha+2}}
$$
which in turn equals
$$
\alpha^3 \lambda \int \frac{\rho(y)\,dy}{|x-y|^{\alpha+2}} \int \frac{\rho(w)}{|x-w|^{\alpha+2}} \left| \int \frac{\rho(z) (x-z)\wedge(x-w)}{|x-w|^{\alpha+2}} \, dz\right|^2 \, dw.
$$

In case (1) ($2<p<n$ and $q\ge 1$) we have $\lambda\le 0$ and $\alpha>0$. In this case, $I_{1,2}+I_2$ is nonpositive.

In case (2) ($p>n$ and $0<q\le 1$) we have $\lambda\le 0$ and $\alpha<0$. In this case, $I_{1,2}+I_2$ is nonnegative.

We now turn to the expression in \eqref{I11I3}. Substituting the value of $\lambda$ given in \eqref{eq:lambda} yields
\begin{equation*}\begin{split}
&I_{1,1} + I_3 \\
&\quad = \tfrac12\alpha^3(\alpha(p-1)+(p-n)) \iiiint \frac{\langle x-z,x-v\rangle dV}{|x-y|^\alpha |x-z|^{\alpha+2} |x-v|^{\alpha+2} |x-w|^{\alpha+2}} \\
&\qquad + \tfrac12\alpha^3(\alpha(p-1)+(p-n)) \iiiint \frac{\langle x-z,x-v\rangle dV}{|x-y|^{\alpha+2} |x-z|^{\alpha+2} |x-v|^{\alpha+2} |x-w|^{\alpha}} \\
&\quad\qquad + (q-1) \alpha^4 \iiiint \frac{\langle x-y,x-w\rangle \, \langle x-z,x-v \rangle
dV}{|x-y|^{\alpha+2} |x-z|^{\alpha+2} |x-v|^{\alpha+2} |x-w|^{\alpha+2}} \\
\end{split}\end{equation*}
Since
$$
\alpha^3(\alpha(p-1)+(p-n)) + (q-1)\alpha^4 = \alpha^3(\alpha(p-2+q)+(p-n))
$$
we may rewrite this as follows:
\begin{equation*}\begin{split}
&I_{1,1} + I_3 \\
&\quad = \tfrac12\alpha^3(\alpha(p-2+q)+(p-n)) \iiiint \frac{\langle x-z,x-v\rangle dV}{|x-y|^\alpha |x-z|^{\alpha+2} |x-v|^{\alpha+2} |x-w|^{\alpha+2}} \\
&\qquad + \tfrac12\alpha^3(\alpha(p-2+q)+(p-n)) \iiiint \frac{\langle x-z,x-v\rangle dV}{|x-y|^{\alpha+2} |x-z|^{\alpha+2} |x-v|^{\alpha+2} |x-w|^{\alpha}} \\
&\quad\qquad - \tfrac12(q-1) \alpha^4 \iiiint \frac{|y-w|^2 \langle x-z,x-v \rangle dV}{|x-y|^{\alpha+2} |x-z|^{\alpha+2} |x-v|^{\alpha+2} |x-w|^{\alpha+2}}.
\end{split}\end{equation*}

In case (2) we have
$$
0<q\le 1, \quad \alpha<0, \quad \alpha(p-1)+(p-n)\le 0, \quad \mbox{and} \quad \alpha(p-2+q)+(p-n) \le 0.
$$
Hence all three terms are nonnegative.

In case (1) we have
$$
q\ge 1, \quad \alpha>0, \quad \alpha(p-1)+(p-n)\le 0, \quad \mbox{and} \quad \alpha(p-2+q)+(p-n)) \le 0.
$$
Hence all three terms are nonpositive.

Combining all of the preceding conclusions, we see that $\cL_{p,q} u$ is nonnegative or nonpositive as indicated in the statement of the theorem. The proof of Theorem \ref{th:main} is complete. \qed

\

In Theorem \ref{th:main} we considered measures $\rho(y)\,dy$ absolutely continuous with respect to the Lebesgue measure. We now show how to extend the conclusion to Riesz potentials of general Radon measures $\mu$ satisfying the growth condition
\begin{equation}\label{eq:growth}
\int_{\{y:|y| \ge 1\}}\frac{d\mu(y)}{|y|^\alpha} < \infty.
\end{equation}
To this end, we make use of the connection between solutions for the operator $\cL_{p,q}$ and solutions for the $p$-Laplace operator $\triangle_p$. We denote by $I_\alpha(\mu)$ the Riesz potential of $\mu$, defined to be
$$
I_\alpha(\mu)(x) :=  c_{n,\alpha} \int_{\R^n} \frac{d\mu(y)}{|x-y|^{n-\alpha}},
$$
where $c(n,\alpha)$ is defined as in the introduction. As discussed in the introduction, we formulate the conclusion in terms of the $p$-superharmonicity or $p$-subharmonicity of a power of $I_{n-\alpha}(\mu)$.

\begin{theorem}\label{th:radon}
Let $\mu$ be a Radon measure on $\R^n$ satisfying the growth condition \eqref{eq:growth} for some $\alpha$.
\begin{enumerate}
\item If $2\le p<n$ and $q\ge 1$ then
\begin{equation*}
(I_{n-\alpha}(\mu))^m
\end{equation*}
is $p$-superharmonic if
$$
0<\alpha \le \frac{n-p}{p-2+q}.
$$
\item If $p>n$ and $0<q\le 1$ then $(I_{n-\alpha}(\mu))^m$ is $p$-subharmonic if
$$
-\alpha \ge \frac{p-n}{p-2+q}.
$$
If $p=\infty$ we may choose $-\alpha \ge 1$.
\item If $p=n$ and $q>0$ then $(\int \log|x-y| \, d\mu(y))^m$ is $n$-superharmonic.
\end{enumerate}
\end{theorem}

Recall that $m$ is the parameter defined in \eqref{m}.

\begin{proof}
We follow a line of reasoning similar to that used by Lindqvist and Manfredi (pp.\ 137--138 in \cite{lm:remarkable}). We regularize $\mu$ by convolving with the heat kernel. For $t>0$ define
$$
\rho_t(y) = \frac{1}{(4\pi t)^{n/2}} \int e^{-\tfrac{|y-\xi|^2}{4t}} \, d\mu(\xi)
$$
and we consider
$$
u_t = I_{n-\alpha}(\rho_t).
$$
For simplicity we restrict attention to the case when
$$
2\le p < n \qquad \mbox{and} \qquad q \ge 1 \qquad \mbox{and} \qquad \alpha = \frac{n-p}{p-2+q}.
$$
By Theorem \ref{th:main}, $u_t$ is a $C^2$ supersolution for the operator $-\cL_{p,q}$. Alternatively, as previously discussed, $(u_t)^m$ is a supersolution for the $p$-Laplace operator, i.e., $(u_t)^m$ is $p$-superharmonic. Passing to a sequence $(t_k)$ with $t_k\searrow 0$, we have that $(u_{t_k})^m \to u^m$ a.e. By general results from the nonlinear potential theory of the $p$-Laplacian (see, for instance, Theorem 1.17 in \cite{km:degenerate}), the limit function $u^m$ is $p$-superharmonic. The remaining cases are handled similarly. This completes the proof.
\end{proof}

\begin{remark}\label{rem:h-type}
Our setting in this paper has been Euclidean space. We expect that similar conclusions should hold for quasilinear subelliptic equations of the form
\begin{equation}\label{step-two-groups}
\diver_H(u^{q-1}|\nabla_H u|^{p-2}\nabla_H u) = 0
\end{equation}
in step two Carnot groups of Heisenberg type, for instance, in the Heisenberg group itself. See \cite{gt:h-type} for the special case of the $p$-Laplace equation
$$
\diver_H(|\nabla_H u|^{p-2}\nabla_H u) = 0.
$$
Here $\nabla_H u$ denotes the horizontal gradient and $\diver_H V$ denotes the horizontal divergence of a horizontal vector field. We leave the analysis of \eqref{step-two-groups} in Heisenberg-type Carnot groups to a future paper.
\end{remark}

\bibliographystyle{acm}
\bibliography{references}
\end{document}